\documentclass{amsart}
\usepackage{amsfonts,amssymb,amsmath,amsthm}
\usepackage{url}
\usepackage{enumerate}
\usepackage{hyperref}
\newtheorem{theorem}{Theorem}[section]
\newtheorem{lemma}[theorem]{Lemma}

\newtheorem{corollary}[theorem]{Corollary}
\theoremstyle{definition}
\newtheorem{definition}[theorem]{Definition}
\newtheorem{remark}[theorem]{Remark}
\newtheorem{example}[theorem]{Example}
\numberwithin{equation}{section}

\begin{document}
\title[Geodesics and $F$-geodesics on tangent bundle with $\varphi$-Sasaki metric\ldots]{Geodesics and $F$-geodesics on tangent bundle with $\varphi$-Sasaki metric over para-K\"{a}hler-Norden manifold}
\author{Abderrahim ZAGANE$^{1}$}
\address{RELIZANE University, Faculty of Science and Technology, Department of 
Mathematics, 48000, RELIZANE-ALGERIA}
\email{Zaganeabr2018@gmail.com}
\begin{abstract}
In this paper, we investigate some geodesics and $F$-geodesics problems on tangent bundle and on $\varphi$-unit tangent bundle $T^{\varphi}_{1}M$ equipped with the $\varphi$-Sasaki metric over para-K\"{a}hler-Norden manifold $(M^{2m}, \varphi, g)$. 



\textbf{Corresponding author}: Abderrahim ZAGANE 
\end{abstract}

\maketitle

\section{\protect\bigskip Introduction}

On the tangent bundle of a Riemannian manifold one can define natural Riemannian metrics. Their construction makes use of the Levi-Civita connection. Among them, the so called Sasaki metric \cite{Sas} is of particular interest. That is why the geometry of tangent bundle equipped with the Sasaki metric has been studied by many authors such as  Yano and Ishihara \cite{Y.I}, Dombrowski  \cite{Dom},  Salimov  and his collaborators \cite{S.G.A} etc. The rigidity of Sasaki metric has incited some researchers to construct and study other metrics on tangent bundle. This is the reason why they have attempted to search for different metrics on the tangent bundle which are different deformations of the Sasaki metric. Among them, we mention the Cheeger-Gromoll metric \cite {M.T}, the Mus-Sasaki Metric \cite {Z.D2} and the Berger-type deformed Sasaki metric \cite {A.S.G,Yam}. The geometry of tangent bundle remains a rich area of research in differential geometry to this day.\quad

Geodesy on the tangent bundle has been studied by many authors. In particular the oblique geodesics, non-vertical geodesics and their projections onto the base manifold.  Sasaki \cite{Sas2} and Sato \cite{Sat} gave a complete description of the curves and vector fields along them which generated non-vertical geodesics on  the tangent bundle and on  unit  the tangent bundle respectively. They proved that the projected curves have constant  geodesic curvatures (Frenet curvatures). Nagy \cite{Nag} generalized these results to the case of locally symmetric base manifold. Yampolsky \cite{Yam} also did the same studies on the tangent bundle and on  unit  the tangent bundle with the Berger-type deformed Sasaki metric over K\"{a}hlerian manifold, in the cases of locally symmetric base manifold and of the constant holomorphic curvature base manifold. Also, we refer to \cite{D.Z3,S.G.A,S.K,Zag14}. \quad

The notion of $F$-planar curves generalizes the magnetic curves and implicitly the geodesics (see \cite{H.M, M.S}), but  the notion of $F$-geodesic, which is slightly different from that of $F$-planar curve \cite{B.D}. Recently, in mathematics literature, a series of papers on magnetic curves,  $F$-planar curves  and  $F$-geodesic (see \cite{D.G.K,D.I.M.N,Nis}).

In previous works, \cite{Zag7,Zag21}, we proposed the $\varphi$-Sasaki metric on the tangent bundle over para-K\"{a}hler-Norden manifold $(M^{2m}, \varphi, g)$, where we studied respectively the para-K\"{a}hler-Norden properties on the tangent bundle and then Geometry of $\varphi$-Sasaki metric on tangent bundle. In this paper, we investigate some geodesics and $F$-geodesics properties for the $\varphi$-Sasaki metric on the tangent bundle and on $\varphi$-unit tangent bundle. Firstly, we study the geodesics on $\varphi$-unit tangent bundle with respect to the $\varphi$-Sasaki metric, where  we establish  necessary and sufficient conditions under which a curve be a geodesic with respect to this metric (Theorem \ref{th_2}, Corollary\ref{co_1} and Corollary\ref{co_2}), we then discuss the Frenet curvatures of the projected of the non-vertical geodesic (Theorem \ref{th_3}, Theorem \ref{th_4}, Corollary\ref{co_3} and Theorem \ref{th_5}). Secondly, we study the $F$-geodesics and $F$-planar curves on tangent bundle with respect to the $\varphi$-Sasaki metric (Theorem \ref{th_8}, Theorem \ref{th_9} and Theorem \ref{th_10}). Finally, we study the $F$-geodesics and $F$-planar curves on the $\varphi$-unit tangent bundle with respect to the $\varphi$-Sasaki metric (Theorem \ref{th_11}, Theorem \ref{th_12} and Theorem \ref{th_13}).

\section{Preliminaries}

Let $TM$ be the tangent bundle over an $m$-dimensional Riemannian manifold $(M^{m},g)$ and  the natural projection $ \pi: TM \rightarrow M$. A local chart $(U,x^{i})_{i=\overline{1,m}}$ on $M$ induces a local chart  $(\pi^{-1}(U),x^{i},\xi^{i})_{i=\overline{1,m}}$ on $TM$. Denote by $\Gamma_{ij}^{k}$ the Christoffel symbols of $g$ and by $\nabla$ the Levi-Civita connection of $g$.\qquad

The Levi Civita connection $\nabla$ defines a direct sum decomposition
\begin{eqnarray*}
T_{(x,\xi)}TM=V_{(x,\xi)}TM\oplus H_{(x,\xi)}TM.
\end{eqnarray*}
of the tangent bundle to $TM$ at any $(x,\xi)\in TM$ into vertical subspace
\begin{eqnarray*}
V_{(x,\xi)}TM&=&Ker(d\pi_{(x,\xi)})=\{a^{i}\frac{\partial}{\partial \xi^{i}}|_{(x,\xi)},\, a^{i}\in\mathbb{R}\},
\end{eqnarray*}
and the horizontal subspace
\begin{eqnarray*}
H_{(x,\xi)}TM&=&\{a^{i}\frac{\partial}{\partial x^{i}}|_{(x,\xi)}-a^{i}\xi^{j}\Gamma_{ij}^{k}\frac{\partial}{\partial \xi^{k}}|_{(x,\xi)},\, a^{i}\in \mathbb{R}\}.
\end{eqnarray*}.

Let $Z=Z^{i}\frac{\partial}{\partial x^{i}}$ be a local vector field on $M$. The vertical and the horizontal lifts of $Z$ are defined by
\begin{eqnarray*}
{}^{V}\!Z&=& Z^{i}\frac{\partial}{\partial \xi^{i}},\\
{}^{H}\!Z&=&Z^{i}(\frac{\partial}{\partial x^{i}}- \xi^{j}\Gamma_{ij}^{k}\frac{\partial}{\partial \xi^{k}}).
\end{eqnarray*}
We have ${}^{H}\!(\frac{\partial}{\partial x^{i}})=\frac{\partial}{\partial x^{i}}- \xi^{j}\Gamma_{ij}^{k}\frac{\partial}{\partial \xi^{k}}$ and ${}^{V}\!(\frac{\partial}{\partial x^{i}})=\frac{\partial}{\partial \xi^{i}}$, then $({}^{H}\!(\frac{\partial}{\partial x^{i}}),{}^{V}\!(\frac{\partial}{\partial x^{i}}))_{i=\overline{1,m}}$ is a local adapted frame on $TTM$.

\section{$\varphi$-Sasaki metric}
An almost product structure $\varphi$ on a manifold $M$ is a $(1,1)$-tensor field on $M$ such that $\varphi^{2}=id_{M}$, $\varphi\neq \pm id_{M}$ ($id_{M}$ is the identity tensor field of type $(1,1)$ on $M$). The pair $(M,\varphi)$ is called an almost product manifold.\quad

An almost para-complex manifold is an almost product manifold $(M,\varphi)$,  such that the two eigenbundles $TM^{+}$ and $TM^{-}$ associated to the two eigenvalues $+1$ and $-1$ of $\varphi$, respectively, have the same rank. Note that the dimension of an almost par-acomplex manifold is necessarily even \cite{C.F.G}.\quad

An almost para-complex structure $\varphi$ is integrable if the Nijenhuis tensor:
\begin{eqnarray*}
N_{\varphi}(X,Y)=[\varphi X,\varphi Y]-\varphi[X,\varphi Y]-\varphi[\varphi X,Y]+[X,Y]	
\end{eqnarray*}
vanishes identically on $M$. On the other hand, in order that an almost para-complex structure be integrable, it is necessary and sufficient that we can introduce a torsion free linear connection $\nabla$ such that $\nabla\varphi=0$ \cite{S.I.E}.\quad

An almost para-complex Norden manifold $(M^{2m}, \varphi, g)$ is a $2m$-dimensional differentiable manifold $M$ with an almost para-complex structure $\varphi$ and a Riemannian metric $g$ such that:
\begin{eqnarray*}
g(\varphi X, Y)= g(X, \varphi Y)&\Leftrightarrow& g(\varphi X, \varphi Y) = g(X, Y),
\end{eqnarray*}
for any vector fields $X$ and $Y$  on $M$, in this case $g$ is called a pure metric with respect to $\varphi$ or para-Norden metric (B-metric)\cite{S.I.E}.\quad

Also note that 
\begin{eqnarray}\label{twin}
G(X,Y) = g(\varphi X, Y),
\end{eqnarray}
is a bilinear, symmetric tensor field of type $(0, 2)$ on $(M, \varphi)$ and pure with respect to the paracomplex structure $\varphi$, which is called the twin (or dual) metric of $g$, and it plays a role similar to the K\"{a}hler form in Hermitian Geometry.
Some properties of twin Norden metric are investigated in \cite{I.S,S.I.E}.\quad

A para-K\"{a}hler-Norden manifold is an almost para-complex Norden manifold $(M^{2m}, \varphi, g)$ such that $\varphi$ is integrable i.e. $\nabla \varphi = 0$, where $\nabla$ is the Levi-Civita connection of $g$ \cite{S.G.I,S.I.E}.\quad

It is well known that if $(M^{2m}, \varphi, g)$ is a para-K\"{a}hler-Norden manifold, the Riemannian curvature tensor is pure \cite{S.I.E}. 

\begin{definition}\label{def_0}\cite{Zag7}
Let  $(M^{2m}, \varphi, g)$ be a para-K\"{a}hler-Norden manifold. On the tangent bundle $TM$, we define a  $\varphi$-Sasaki metric noted $g^{\varphi}$ by
\begin{eqnarray}
(1)\quad g^{\varphi}({}^{H}\!X,{}^{H}\!Y)_{(x,\xi)} &=&g_{x}(X,Y),\nonumber\\
(2)\quad g^{\varphi}({}^{H}\!X,{}^{V}\!Y)_{(x,\xi)} &=& 0,\nonumber\\
(3)\quad g^{\varphi}({}^{V}\!X,{}^{V}\!Y)_{(x,\xi)} &=& g_{x}(X,\varphi Y)=G_{x}(X, Y),\nonumber
\end{eqnarray}
for any vector fields $X$ and $Y$  on $M$ and $(x,\xi)\in TM$, where $G$ is the twin Norden metric of $g$ defined by \eqref{twin}.
\end{definition}

\begin{theorem}\label{th_0}\cite{Zag7}
Let $(M^{2m}, \varphi, g)$ be a para-K\"{a}hler-Norden manifold and $TM$ its tangent bundle equipped with the $\varphi$-Sasaki metric $g^{\varphi}$. If  $\nabla$ $($resp $\widetilde{\nabla}$$)$ denote the Levi-Civita connection of $(M^{2m}, \varphi, g)$ $($resp  $(TM,g^{\varphi})$ $)$,  then we have
\begin{eqnarray*}
(1)\quad (\widetilde{\nabla}_{{}^{H}\!X}{}^{H}\!Y)_{(x,\xi)}&=&{}^{H}\!(\nabla_{X}Y)_{(x,\xi)}-\frac{1}{2}{}^{V}\!(R_{x}(X,Y)\xi),\\
(2)\quad (\widetilde{\nabla}_{{}^{H}\!X}{}^{V}\!Y)_{(x,\xi)}&=&{}^{V}\!(\nabla_{X}Y)_{(x,\xi)}+\dfrac{1}{2}{}^{H}\!(R_{x}(\varphi \xi,Y)X),\\
(3)\quad (\widetilde{\nabla}_{{}^{V}\!X}{}^{H}\!Y)_{(x,\xi)}&=&\dfrac{1}{2}{}^{H}\!(R_{x}(\varphi \xi,X)Y),\\
(4)\quad (\widetilde{\nabla}_{{}^{V}\!X}{}^{V}\!Y)_{(x,\xi)}&=&0,
\end{eqnarray*}
for all vector fields $X$ and $Y$ on $M$ and $(x,\xi)\in TM$, where $R$ denote the curvature tensor of $(M^{2m}, \varphi, g)$.
\end{theorem}

The $\varphi$-unit tangent sphere bundle  over a para-K\"{a}hler-Norden manifold $(M^{2m}, \varphi, g)$, is the hypersurface 
\begin{eqnarray*}
T^{\varphi}_{1}M=\big\{(x,\xi)\in TM,\,g(\xi,\varphi \xi)=1\big\}.
\end{eqnarray*}
The unit normal vector field to $T^{\varphi}_{1}M$ is given  by
$$\mathcal{N}={}^{V}\!\xi.$$ 

The tangential lift ${}^{T}\!X$ with respect to $g^{\varphi}$ of a vector $X\in T_{x}M$ to $(x,\xi)\in T^{\varphi}_{1}M$ as the tangential projection of the vertical lift of $X$
to  $(x,\xi)$ with respect to $\mathcal{N}$, that is
\begin{equation*}
{}^{T}\!X={}^{V}\!X-g^{\varphi}_{(x,\xi)}({}^{V}\!X,\mathcal{N}_{(x,\xi)})\mathcal{N}_{(x,\xi)}={}^{V}\!X-g_{x}(X,\varphi \xi){}^{V}\!\xi.
\end{equation*}

The tangent space $T_{(x,\xi)}T^{\varphi}_{1}M$ of $T^{\varphi}_{1}M$ at $(x,\xi)$ is given by
\begin{eqnarray*}
T_{(x,\xi)}T^{\varphi}_{1}M= \{{}^{H}\!X+{}^{T}\!Y\,/\, X \in T_{x}M, Y\in  \xi^{\bot}\subset T_{x}M\}.
\end{eqnarray*}
where $\xi^{\bot}=\big\{Y\in T_{x}M,\,g(Y,\varphi \xi)=0\big\}$, see \cite{Zag21}.

\begin{theorem}\label{th_1}\cite{Zag21}
Let $(M^{2m}, \varphi, g)$ be a para-K\"{a}hler-Norden manifold and $T^{\varphi}_{1}M$ its $\varphi$-unit tangent bundle equipped with the $\varphi$-Sasaki metric. If $\widehat{\nabla}$ denote the Levi-Civita connection of $\varphi$-Sasaki metric on $T^{\varphi}_{1}M$, then we have the following formulas.
\begin{eqnarray*}
1.\,\widehat{\nabla}_{{}^{H}\!X}{}^{H}\!Y&=&{}^{H}\!(\nabla_{X}Y)-\frac{1}{2}{}^{T}\!(R(X,Y)\xi),\\
2.\,\widehat{\nabla}_{{}^{H}\!X}{}^{T}\!Y&=&{}^{T}\!(\nabla_{X}Y)+\dfrac{1}{2}{}^{H}\!\left(R(\varphi \xi,Y)X\right),\\
3.\,\widehat{\nabla}_{{}^{T}\!X}{}^{H}\!Y&=&\dfrac{1}{2}{}^{H}\!\left(R(\varphi \xi,X)Y\right),\\
4.\,\widehat{\nabla}_{{}^{T}\!X}{}^{T}\!Y&=&-g(Y,\varphi \xi){}^{T}\!X,
\end{eqnarray*}
for all vector fields $X$ and $Y$ on $M$, where $\nabla$ is the Levi-Civita connection and $R$ is its curvature tensor of $(M^{2m}, \varphi, g)$.
\end{theorem}

\section{Geodesics on $\varphi$-unit tangent bundle with the $\varphi$-Sasaki metric}
Let $\Gamma =(\gamma(t),\xi(t))$ be a naturally parameterized curve on the tangent bundle $TM$ (i.e. $t$ is an arc length parameter on $\Gamma$), where $\gamma$ is a curve on $M$ and $\xi$ is a vector field along this curve. Denote $\gamma_{t}^{\prime}=\frac{d\,x}{d\,t}$, $\gamma_{t}^{\prime\prime}=\nabla_{\gamma_{t}^{\prime}}\gamma_{t}^{\prime}$, $\xi_{t}^{\prime}=\nabla_{\gamma_{t}^{\prime}}\xi$,  $\xi_{t}^{\prime\prime}=\nabla_{\gamma_{t}^{\prime}}\xi_{t}^{\prime}$ and $\Gamma_{t}^{\prime}=\frac{d\,\Gamma}{d\,t}$. Then
\begin{equation}\label{eq_1}
\Gamma_{t}^{\prime}={}^{H}\!\gamma_{t}^{\prime} + {}^{V}\!\xi_{t}^{\prime}.
\end{equation}

\begin{lemma}
Let $(M^{2m}, \varphi, g)$ be a para-K\"{a}hler-Norden manifold, $T^{\varphi}_{1}M$ its $\varphi$-unit tangent bundle equipped with the $\varphi$-Sasaki metric  and $\Gamma =(\gamma(t),\xi(t))$  be a curve on $T^{\varphi}_{1}M$. Then we have
\begin{equation}\label{eq_2}
\Gamma_{t}^{\prime} = {}^{H}\!\gamma_{t}^{\prime} + {}^{T}\!\xi_{t}^{\prime},
\end{equation}
\end{lemma}

\begin{proof} Using $\eqref{eq_1}$, we have 
\begin{eqnarray*}
\Gamma_{t}^{\prime} &=& {}^{H}\!\gamma_{t}^{\prime} + {}^{V}\!\xi_{t}^{\prime}={}^{H}\!\gamma_{t}^{\prime} + {}^{T}\!\xi_{t}^{\prime}+g(\xi_{t}^{\prime},\varphi \xi){}^{V}\!u.
\end{eqnarray*}	 
Since $\Gamma =(\gamma(t),\xi(t))\in T^{\varphi}_{1}M$ then $g(\xi,\varphi \xi)=1$, on the other hand
\begin{eqnarray*}
0&=&\gamma_{t}^{\prime}g(\xi,\varphi \xi)=2g(\xi_{t}^{\prime},\varphi \xi),
\end{eqnarray*}
i.e.
\begin{eqnarray}\label{eq_3}
\quad g(\xi_{t}^{\prime},\varphi \xi)=0.
\end{eqnarray}
Hence, the proof of the lemma is completed.
\end{proof}
From \eqref{eq_2}, we have
\begin{eqnarray}\label{eq_4}
1=|\gamma_{t}^{\prime}|^{2}+ g(\xi_{t}^{\prime},\varphi \xi_{t}^{\prime}),
\end{eqnarray}
where  $|\,.\,|$ mean the norm of vectors with respect to the $(M^{2m}, \varphi, g)$.

\begin{theorem}\label{th_2}
Let $(M^{2m}, \varphi, g)$ be a para-K\"{a}hler-Norden manifold, $T^{\varphi}_{1}M$ its $\varphi$-unit tangent bundle equipped with the $\varphi$-Sasaki metric  and $\Gamma =(\gamma(t),\xi(t))$  be a curve on $T^{\varphi}_{1}M$. Then $\Gamma$ is a geodesic on $T^{\varphi}_{1}M$ if and only if
\begin{equation}\label{eq_5}
\left\{
\begin{array}{lll}
\gamma_{t}^{\prime\prime}&=&R(\xi_{t}^{\prime},\varphi \xi)\gamma_{t}^{\prime}\\
\xi_{t}^{\prime\prime}&=&g(\xi_{t}^{\prime\prime}, \varphi \xi) \xi
\end{array}
\right.
\end{equation}
Moreover, 
\begin{equation}\label{eq_6}
\left\{
\begin{array}{lll}
|\gamma_{t}^{\prime}|&=&\sqrt{1-\rho^{2}}\\
g(\xi_{t}^{\prime\prime}, \varphi \xi)&=&-g(\xi_{t}^{\prime}, \varphi \xi_{t}^{\prime})=-\rho^{2}
\end{array}
\right.	
\end{equation}
where $\rho = const$ and $0\leq \rho \leq1$.
\end{theorem}

\begin{proof} Using formula $\eqref{eq_2}$and Theorem \ref{th_1}, we compute the derivative $\widehat{\nabla}_{\Gamma_{t}^{\prime}}\Gamma_{t}^{\prime}$. 
\begin{eqnarray*}
\widehat{\nabla}_{\Gamma_{t}^{\prime}}\Gamma_{t}^{\prime} & = &\widehat{\nabla}_{\displaystyle({}^{H}\!\gamma_{t}^{\prime} + {}^{T}\!\xi_{t}^{\prime})}({}^{H}\!\gamma_{t}^{\prime} + {}^{T}\!\xi_{t}^{\prime}) \\
& = &\widehat{\nabla}_{\displaystyle{}^{H}\gamma_{t}^{\prime}}{}^{H}\gamma_{t}^{\prime} +\widehat{\nabla}_{\displaystyle{}^{H}\!\gamma_{t}^{\prime}}{}^{T}\!\xi_{t}^{\prime}+\widehat{\nabla}_{{}^{T}\!\xi_{t}^{\prime}}{}^{H}\!\gamma_{t}^{\prime}+\widehat{\nabla}_{{}^{T}\!\xi_{t}^{\prime}}{}^{T}\!\xi_{t}^{\prime} \\
&=& {}^{H}\!\gamma_{t}^{\prime\prime}+{}^{H}\!(R(\varphi \xi,\xi_{t}^{\prime})\gamma_{t}^{\prime})+{}^{T}\!\xi_{t}^{\prime\prime}\\
&=&{}^{H}\!\big(\gamma_{t}^{\prime\prime}+R(\varphi \xi, \xi_{t}^{\prime})\gamma_{t}^{\prime}\big)+{}^{V}\!(\xi_{t}^{\prime\prime}-g(\xi_{t}^{\prime\prime}, \varphi \xi)\xi).
\end{eqnarray*}
If we put $\widehat{\nabla}_{\Gamma_{t}^{\prime}}\Gamma_{t}^{\prime}$ equal to zero, we find $\eqref{eq_5}$.\\
From $\eqref{eq_3}$ we get, $0=\gamma_{t}^{\prime}g(\xi_{t}^{\prime}, \varphi \xi)=g(\xi_{t}^{\prime\prime}, \varphi \xi)+g(\xi_{t}^{\prime}, \varphi \xi_{t}^{\prime})$ then, 
\begin{eqnarray*}
g(\xi_{t}^{\prime\prime}, \varphi \xi)=-g(\xi_{t}^{\prime}, \varphi \xi_{t}^{\prime}),
\end{eqnarray*}
Using $\eqref{eq_3}$ and the second equation of $\eqref{eq_5}$, we find, 
$$\gamma_{t}^{\prime}g(\xi_{t}^{\prime}, \varphi \xi_{t}^{\prime})=2g(\xi_{t}^{\prime\prime}, \varphi \xi_{t}^{\prime})=2g(\xi_{t}^{\prime\prime}, \varphi \xi)g(\xi, \varphi \xi_{t}^{\prime})=0,$$ 
then $g(\xi_{t}^{\prime}, \varphi \xi_{t}^{\prime})=\kappa=const$, from $\eqref{eq_4}$, we find,  $0\leq \kappa \leq1$,\\
hence  $|\gamma_{t}^{\prime}|=\sqrt{1-\rho^{2}}$ and $g(\xi_{t}^{\prime\prime}, \varphi \xi)=-g(\xi_{t}^{\prime}, \varphi \xi_{t}^{\prime})=-\rho^{2}$, where $\rho^{2}=\kappa$.
\end{proof}

\begin{remark}\label{rem_1}	According to \eqref{eq_6}, the geodesics $\Gamma =(\gamma(t),\xi(t))$ of $T^{\varphi}_{1}M$ can be splitted naturally into 3 classes, namely,
\item(1) horizontal geodesics, if $\rho=0$, from \eqref{eq_6}, $|\gamma_{t}^{\prime}|=1$, then from \eqref{eq_4}, we have $\xi_{t}^{\prime}=0$ i.e. $\Gamma$ is generated by parallel vector fields $\xi$ along the geodesics $\gamma$ on the base manifold,
\item(2) vertical geodesics, if $\rho=1$, from \eqref{eq_6}, $|\gamma_{t}^{\prime}|=0$, then $\gamma(t)$ is a constant i.e. $\Gamma$ is geodesic in Euclidean space, (on a fixed fiber), their equations are\, $\xi_{t}^{\prime\prime}=\rho^{2}\xi$. 
\item(3) umbilical (oblique) geodesics corresponding to $0 < \rho < 1$, In this case, $\Gamma$ can be regarded as a vector field $\xi\neq 0$ along the curve $\gamma$. see \cite{Y.S}.
\end{remark}

A curve $\Gamma= (\gamma(t), \xi(t))$ on $TM$ is said to be a horizontal lift of the curve
$\gamma$ on $M$ if and only if $\xi_{t}^{\prime}=0$ \cite{Y.I}.

In general, the horizontal lift $\Gamma= (\gamma(t), \xi(t))$ of the curve $\gamma$ on $M$ does not belong to $T^{\varphi}_{1}M$, we have \;$\xi_{t}^{\prime}=0$, then $0=2g(\xi_{t}^{\prime}, \varphi \xi)=\gamma_{t}^{\prime}g(\xi, \varphi \xi)=0$, hence $g(\xi, \varphi \xi)=const\neq 1$. If $\Gamma\in T^{\varphi}_{1}M$, we get the following Corollary

\begin{corollary}\label{co_1} Let $(M^{2m}, \varphi, g)$ be a para-K\"{a}hler-Norden manifold and  $T^{\varphi}_{1}M$ its $\varphi$-unit tangent bundle equipped with the $\varphi$-Sasaki metric. If  $\Gamma= (\gamma(t), \xi(t))$ is a horizontal lift of $\gamma$ and $\Gamma\in T^{\varphi}_{1}M$, then $\Gamma$ is a geodesic on $T^{\varphi}_{1}M$ if and only if  $\gamma$ is a geodesic on $M$. 	
\end{corollary}

The curve $\Gamma = (\gamma(t), \gamma_{t}^{\prime}(t))$ is called a natural lift of the curve $\gamma$ on $TM$ \cite{Y.I}.  Thus, we have

\begin{corollary}\label{co_2} Let $(M^{2m}, \varphi, g)$ be a para-K\"{a}hler-Norden manifold and $T^{\varphi}_{1}M$ its $\varphi$-unit tangent bundle equipped with the $\varphi$-Sasaki metric. If  $\Gamma= (\gamma(t), \gamma_{t}^{\prime}(t))$ is a natural lift of $\gamma$ and $\Gamma\in T^{\varphi}_{1}M$, then  $\Gamma$ is a geodesic on $T^{\varphi}_{1}M$ if and only if  $\gamma$ is a geodesic on $M$.
\end{corollary}

\begin{remark}\label{re_1}
As a reminder, note that locally we have:
\begin{eqnarray}\label{eq_I}
\gamma_{t}^{\prime\prime}=\sum_{l=1}^{2m} (\dfrac{d^{2}\gamma^{l}}{dt^{2}}+\sum_{i,j=1}^{2m}\dfrac{d\gamma^{i}}{dt}\dfrac{d\gamma^{j}}{dt}\Gamma_{ij}^{l})\frac{\partial}{\partial x^{l}},  
\end{eqnarray}
and
\begin{eqnarray}\label{eq_II}
\xi_{t}^{\prime}=\sum_{l=1}^{2m} (\dfrac{d\xi^{l}}{dt}+\sum_{i,j=1}^{2m}\frac{d\gamma^{j}}{dt}\xi^{i}\Gamma_{ij}^{l})\frac{\partial}{\partial x^{l}}.  
\end{eqnarray}
\end{remark} 

\begin{example}
Let $(\mathbb{R}^{2},\varphi,g)$ be a para-K\"{a}hler-Norden manifold such that
$$g= e^{2x}dx^{2}+e^{2y}dy^{2},\quad \varphi=\left( \begin{array}{ccc}
0 & e^{y-x}\\
e^{x-y} & 0  
\end{array} \right).$$
The non-null Christoffel symbols of the Riemannian connection are:
$$\Gamma_{11}^{1} = \Gamma_{22}^{2} = 1.$$
$1)$ Let  $\gamma$ be a curve such that $\gamma(t)=(x(t), y(t))$, from \eqref{eq_I}, the geodesics $\gamma$ such that $\gamma(0)= (a, b)\in \mathbb{R}^{2}$ and  $\gamma_{t}^{\prime}(0)= (\lambda,\eta)\in \mathbb{R}^{\ast}_{+}\times \mathbb{R}^{\ast}_{+}$ satisfies the system of differential equations, 
\begin{eqnarray*}
\frac{d^{2}\gamma^{l}}{dt^{2}}
+\sum_{i,j=1}^{2}\frac{d\gamma^{i}}{dt}\frac{d\gamma^{j}}{dt}\Gamma_{ij}^{l}=0&\Leftrightarrow& \left\{\begin{array}{lll}
\dfrac{d^{2}x}{dt^{2}} + (\dfrac{dx}{dt})^{2}=0\\\\
\dfrac{d^{2}y}{dt^{2}} + (\dfrac{dy}{dt})^{2}=0
\end{array} \right.\\&\Leftrightarrow& \left\{\begin{array}{lll}
x(t)=a+\ln( 1+\lambda t)\\\\
y(t)=b+\ln( 1+\eta t)
\end{array} \right.
\end{eqnarray*}
Hence, $$\gamma_{t}^{\prime}(t)= \dfrac{\lambda}{1 +\lambda t} \dfrac{\partial}{\partial x}+\dfrac{\eta}{1 +\eta t}\dfrac{\partial}{\partial y},\;\gamma(t)= (a+\ln( 1+\lambda t),b+\ln( 1+\eta t)).$$
On the other hand we have $g(\gamma_{t}^{\prime}, \varphi\gamma_{t}^{\prime})= 2\lambda\eta e^{a+b}$, then for  $\lambda\eta= \dfrac{1}{2e^{a+b}}$,\\
become $\Gamma_{1} =(\gamma(t),\gamma_{t}^{\prime}(t))\in T^{\varphi}_{1}\mathbb{R}^{2}$.\\
Hence from Corollary $\ref{co_2}$, the curve $\Gamma_{1}$ is a geodesic on $T^{\varphi}_{1}\mathbb{R}^{2}$.\\
$2)$ If $ \Gamma_{2} =(\gamma(t),\xi(t))$ is horizontal lift of $\gamma$, such that $\xi(t) =(u(t),v(t))$ i.e. $\xi_{t}^{\prime}=0$, from \eqref{eq_II}, we have 
$$\dfrac{d\xi^{l}}{dt}+\sum_{i,j=1}^{2}\frac{dx^{j}}{dt}\xi^{i}\Gamma_{ij}^{l}=0 \Leftrightarrow \left\{\begin{array}{lll}
\dfrac{du}{dt} + \dfrac{dx}{dt}u=0\\\\
\dfrac{dv}{dt} + \dfrac{dy}{dt}v=0
\end{array} \right.\Leftrightarrow \left\{\begin{array}{lll}
u(t)=\dfrac{h_{1}}{1+\lambda t}\\\\
v(t)=\dfrac{h_{2}}{1+\eta t}
\end{array} \right.$$
Hence $\xi(t)=\dfrac{h_{1}}{1+\lambda t}\dfrac{\partial}{\partial x}+\dfrac{h_{2}}{1+\eta t}\dfrac{\partial}{\partial y}$, where  $h_{1}, h_{2}\in \mathbb{R}.$\\
But $g(\xi^{\prime}, \varphi\xi)= 2h_{1}h_{2} e^{a+b}$, then for  $h_{1}h_{2}= \dfrac{1}{2e^{a+b}}$,\\ 
become $\Gamma_{2} =(\gamma(t),\xi(t))\in T^{\varphi}_{1}\mathbb{R}^{2}$.\\
Hence from Corollary $\ref{co_1}$, the curve $\Gamma_{2}$ is a geodesic on $T^{\varphi}_{1}\mathbb{R}^{2}$.
\end{example}

Let $\Gamma$ be a curve on $TM$, the cure $\gamma=\pi\circ \Gamma$ is called the projection (projected curve) of the curve $\Gamma$ on $M$. 

\begin{theorem}\label{th_3}
Let $(M^{2m}, \varphi, g)$ be a locally symmetric para-K\"{a}hler-Norden manifold, $T^{\varphi}_{1}M$ its $\varphi$-unit tangent bundle equipped with the $\varphi$-Sasaki metric and $\Gamma$ be a non-vertical geodesic on $T^{\varphi}_{1}M$, then all Frenet curvatures of the projected curve $\gamma=\pi\circ \Gamma$ are constants.
\end{theorem}

\begin{proof} Using the first equation of $\eqref{eq_5}$, we have\; $\gamma_{t}^{\prime\prime}=R(\xi_{t}^{\prime}, \varphi \xi)\gamma_{t}^{\prime}.$\\
It is easy to see that 
\begin{eqnarray*}
\gamma_{t}^{\prime}g(\gamma_{t}^{\prime}, \gamma_{t}^{\prime})=2g(\gamma_{t}^{\prime\prime}, \gamma_{t}^{\prime})=2g(R(\xi_{t}^{\prime}, \varphi \xi)\gamma_{t}^{\prime}, \gamma_{t}^{\prime})=0,
\end{eqnarray*} 
hence,\; $|\gamma_{t}^{\prime}|= const$.\\ 
Calculate the third derivative, we get 
\begin{eqnarray*}
\gamma_{t}^{\prime\prime\prime}&=&(\nabla_{\gamma_{t}^{\prime}}R)(\xi_{t}^{\prime}, \varphi \xi)\gamma_{t}^{\prime}+R(\xi_{t}^{\prime\prime}, \varphi \xi)\gamma_{t}^{\prime}+R(\xi_{t}^{\prime}, \varphi \xi_{t}^{\prime})\gamma_{t}^{\prime}+R(\xi_{t}^{\prime}, \varphi \xi)\gamma_{t}^{\prime\prime}\\
&=&R(\xi_{t}^{\prime}, \varphi \xi)\gamma_{t}^{\prime\prime}=R^{2}(\xi_{t}^{\prime}, \varphi \xi)\gamma_{t}^{\prime}.
\end{eqnarray*}
Since 
\begin{eqnarray*}
\gamma_{t}^{\prime}g(\gamma_{t}^{\prime\prime}, \gamma_{t}^{\prime\prime})=2g(\gamma_{t}^{\prime\prime\prime}, \gamma_{t}^{\prime\prime})=2g(R(\xi_{t}^{\prime}, \varphi \xi)\gamma_{t}^{\prime\prime}, \gamma_{t}^{\prime\prime})=0,
\end{eqnarray*} 
hence,\; $|\gamma_{t}^{\prime\prime}|= const$.\\ 	
Continuing the process, we obtain 
\begin{equation}\label{eq_7}
\gamma_{t}^{(p+1)}=R(\xi_{t}^{\prime}, \varphi \xi)\gamma_{t}^{(p)}=R^{p}(\xi_{t}^{\prime}, \varphi \xi)\gamma_{t}^{\prime},\quad p\geq1
\end{equation}
and
\begin{equation*}
\gamma_{t}^{\prime}g(\gamma_{t}^{(p)}, \gamma_{t}^{(p)})=2g(\gamma_{t}^{(p+1)},\gamma_{t}^{(p)})=2g(R(\xi_{t}^{\prime}, \varphi \xi)\gamma_{t}^{(p)},\gamma_{t}^{(p)})=0.
\end{equation*}
Thus, we get 
\begin{equation}\label{eq_8}
|\gamma_{t}^{(p)}|=const,\quad p\geq 1.  
\end{equation}
Denote by $s$ an arc length parameter on $\gamma $, i.e. $(|\gamma_{s}^{\prime}|=1)$.\\
Then, $\gamma_{t}^{\prime}=
\gamma_{s}^{\prime}\frac{d\,s}{dt}$, and using $\eqref{eq_6}$, we get
\begin{equation}\label{eq_9}
\frac{d\,s}{dt}=\sqrt{1-\rho^{2}}=const,	
\end{equation}
Let $\nu_{1} = \gamma_{s}^{\prime}$ be the first vector in the Frenet frame $\nu_{1}, \ldots, \nu_{2m-1}$ along $\gamma$ and let $k_{1}, \ldots, k_{2m-1}$ the Frenet curvatures of $\gamma$. Then the Frenet formulas verify
\begin{equation}\label{eq_10}
\left\{
\begin{array}{lll}
(\nu_{1})^{\prime}_{s}&=&k_{1}\nu_{2}\\
(\nu_{i})^{\prime}_{s}&=&-k_{i-1}\nu_{i-1}+k_{i}\nu_{i+1},\quad 2\leq i\leq 2m-2\\
(\nu_{2m-1})^{\prime}_{s}&=&-k_{2m-2}\nu_{2m-2}
\end{array}
\right.
\end{equation}
From $\eqref{eq_6}$, we have 
\begin{eqnarray}\label{eq_11}
\gamma_{t}^{\prime}=\gamma_{s}^{\prime}\frac{d\,s}{dt}=\sqrt{1-\rho^{2}}\,\nu_{1}.
\end{eqnarray}
Using $\eqref{eq_9}$ and the Frenet formulas $\eqref{eq_10}$, we obtain
\begin{eqnarray}\label{eq_12}
\gamma_{t}^{\prime\prime}=\sqrt{1-\rho^{2}}\,(\nu_{1})^{\prime}_{t}=\sqrt{1-\rho^{2}}\,(\nu_{1})^{\prime}_{s}\frac{d\,s}{dt}=(1-\rho^{2})k_{1}\nu_{2}.
\end{eqnarray}
Now $\eqref{eq_8}$ implies $k_{1} =$ const.\\
Next, in a similar way, we have
\begin{eqnarray}\label{eq_13}
\gamma_{t}^{\prime\prime\prime}&=&(1-\rho^{2})k_{1}(\nu_{2})^{\prime}_{t}=(1-\rho^{2})k_{1}(\nu_{2})^{\prime}_{s}\frac{d\,s}{dt}\nonumber\\
&=&-(1-\rho^{2})^{\frac{3}{2}}k_{1}^{2}\nu_{1}+(1-\rho^{2})^{\frac{3}{2}}k_{1}k_{2}\nu_{3}.
\end{eqnarray}
and again $\eqref{eq_8}$ implies $k_{2}=$ const.\\
By continuing the process, we finish the proof.
\end{proof}

\begin{lemma}\label{lem_1}
Let $(M^{2m}, \varphi, g)$ be a para-K\"{a}hler-Norden manifold and $T^{\varphi}_{1}M$ its $\varphi$-unit tangent bundle equipped with the $\varphi$-Sasaki metric. If $\Gamma=(\gamma(t),\xi(t))$ is a curve on $T^{\varphi}_{1}M$, then we have
\item $(1)$ $\Phi=(\gamma(t),\varphi \xi(t))$ is a curve on $T^{\varphi}_{1}M$.
\item $(2)$ $\Phi$ is a geodesic on $T^{\varphi}_{1}M$ if and only if $\Gamma$ is a geodesic on $T^{\varphi}_{1}M$.
\end{lemma}

\begin{proof}	
\item $(1)$ We put $\mu(t)=\varphi \xi(t)$, since $\Gamma =(\gamma(t),\xi(t))\in T^{\varphi}_{1}M$, then  $g(\xi,\varphi \xi)=1$. \\ On the other hand, 	 
$g(\mu,\varphi \mu)=g(\varphi \xi,\varphi(\varphi \xi))=g(\varphi \xi,\xi)=1$ i.e. $$\Phi(t) =(\gamma(t),\mu(t))\in T^{\varphi}_{1}M.$$
\item $(2)$ In a similar way proof of $\eqref{eq_5}$, and using  $\mu_{t}^{\prime}=\varphi \xi_{t}^{\prime}$ and $\mu_{t}^{\prime\prime}=\varphi \xi_{t}^{\prime\prime}$, we have
\begin{eqnarray*}
\widehat{\nabla}_{\Phi_{t}^{\prime}}\Phi_{t}^{\prime} &=& {}^{H}\!(\gamma_{t}^{\prime\prime}+R(\varphi \mu,\mu_{t}^{\prime})\gamma_{t}^{\prime})+{}^{V}\!(\mu_{t}^{\prime\prime}+g(\mu_{t}^{\prime}, \varphi \mu_{t}^{\prime})\mu)\\
&=&{}^{H}\!\big(\gamma_{t}^{\prime\prime}+R(\xi,\varphi \xi_{t}^{\prime})\gamma_{t}^{\prime}\big)+{}^{V}\!(\varphi \xi_{t}^{\prime\prime}+g(\varphi \xi_{t}^{\prime}, \xi_{t}^{\prime})\varphi \xi).
\end{eqnarray*}
Since the Riemannian curvature tensor is pure, we get
\begin{eqnarray*}
\widehat{\nabla}_{\Phi_{t}^{\prime}}\Phi_{t}^{\prime}&=&{}^{H}\!\big(\gamma_{t}^{\prime\prime}+R(\varphi \xi,\xi_{t}^{\prime})\gamma_{t}^{\prime}\big)+{}^{V}\!(\varphi (\xi_{t}^{\prime\prime}+g(\xi_{t}^{\prime}, \varphi \xi_{t}^{\prime}) \xi)),
\end{eqnarray*}
hence,
\begin{eqnarray*}
\widehat{\nabla}_{\Phi_{t}^{\prime}}\Phi_{t}^{\prime}=0 &\Leftrightarrow&\left\{
\begin{array}{lll}
\gamma_{t}^{\prime\prime}&=&-R(\varphi \xi,\xi_{t}^{\prime})\gamma_{t}^{\prime}\\
\varphi \xi_{t}^{\prime\prime}&=&-g(\xi_{t}^{\prime}, \varphi \xi_{t}^{\prime})\varphi \xi
\end{array}\right.\\
&\Leftrightarrow& \left\{\begin{array}{lll}
\gamma_{t}^{\prime\prime}&=&R(\xi_{t}^{\prime},\varphi \xi)\gamma_{t}^{\prime}\\
\xi_{t}^{\prime\prime}&=&-g(\xi_{t}^{\prime}, \varphi \xi_{t}^{\prime}) \xi
\end{array}
\right.\\
&\Leftrightarrow& \widehat{\nabla}_{\Gamma_{t}^{\prime}}\Gamma_{t}^{\prime}=0.
\end{eqnarray*}
\end{proof}

From Theorem \ref{th_3} and Lemma \ref{lem_1}, we have the following theorem 

\begin{theorem}\label{th_4}
Let $(M^{2m}, \varphi, g)$ be a locally symmetric para-K\"{a}hler-Norden manifold, $T^{\varphi}_{1}M$ its $\varphi$-unit tangent bundle equipped with the $\varphi$-Sasaki metric and $\Gamma =(\gamma(t), \xi(t))$ be a non-vertical geodesic on $T^{\varphi}_{1}M$ then, all Frenet curvatures  of the projected curve $\pi\circ \Phi$ are constants, where  $\Phi =(\gamma(t),\varphi \xi(t))$.
\end{theorem}

Now we study the geodesics on $\varphi$-unit tangent bundle with the $\varphi$-Sasaki metric over para-K\"{a}hler-Norden manifold of constant sectional curvature.\\
From Theorem \ref{th_2}, we have the following

\begin{corollary}\label{co_3}
Let $(M^{2m}, \varphi, g)$ be a para-K\"{a}hler-Norden manifold of constant curvature $c\neq0$, $T^{\varphi}_{1}M$ its $\varphi$-unit tangent bundle equipped with the $\varphi$-Sasaki metric  and $\Gamma =(\gamma(t),\xi(t))$  be a curve on $T^{\varphi}_{1}M$. Then $\Gamma$ is a geodesic on $T^{\varphi}_{1}M$ if and only if
\begin{equation}\label{eq_14}
\left\{
\begin{array}{lll}
\gamma_{t}^{\prime\prime}&=&cg(\varphi \xi, \gamma_{t}^{\prime})\xi_{t}^{\prime}-cg(\xi_{t}^{\prime}, \gamma_{t}^{\prime}) \varphi \xi\\
\xi_{t}^{\prime\prime}&=&-g(\xi_{t}^{\prime}, \varphi \xi_{t}^{\prime}) \xi
\end{array}
\right.
\end{equation}
\end{corollary}

\begin{theorem}\label{th_5}
Let $(\mathbb{R}^{2m}, \varphi, <,>)$ be a para-K\"{a}hler-Norden real euclidean space , $T^{\varphi}_{1}\mathbb{R}^{2m}$ its $\varphi$-unit tangent bundle equipped with the $\varphi$-Sasaki metric. Any oblique geodesics $\Gamma =(\gamma(t),\xi(t))$ on $T^{\varphi}_{1}\mathbb{R}^{2m}$ is the following form
\begin{equation}\label{eq_15}
\left\{
\begin{array}{lll}
\gamma(t)&=&c_{1}t+c_{2}\\
\xi(t)&=&c_{3}\cos (\rho t)+c_{4}\sin(\rho t)
\end{array}
\right.
\end{equation}
where $\rho = const$ and $0< \rho <1$ and $c_{1}, c_{2}, c_{3}, c_{4}$ are real constants.
\end{theorem}

\begin{proof}
 In the case of Euclidean space we have $R = 0$, then, using \eqref{eq_5} and \eqref{eq_6} , we get
\begin{equation*}
\left\{
\begin{array}{lll}
\gamma_{t}^{\prime\prime}&=&0\\
\xi_{t}^{\prime\prime}&=&-\rho^{2} \xi
\end{array}
\right.\Leftrightarrow\left\{
\begin{array}{lll}
\gamma(t)&=&c_{1}t+c_{2}\\
\xi(t)&=&c_{3}\cos (\rho t)+c_{4}\sin(\rho t)
\end{array}
\right.
\end{equation*}	
\end{proof}

Define a power of curvature operator $R^{p}(X, Y)$ recurrently in the following way:
$$R^{p}(X, Y)Z=R^{p-1}(X, Y)R(X, Y)Z,$$
for any vector fields $X$ and $Y$  on $M$, where  $p\geq 2$.

\begin{lemma} \label{lem_2}\cite{Y.S} Let $(M,g)$ be a Riemannian manifold of constant
curvature $c$, then we have
\begin{equation*}
R^{p}(X, Y)=\left\{
\begin{array}{lll}
(-b^{2}c^{2})^{k-1}R(X,Y), \;\; \textit{for}\;\; p = 2k-1\\\
(-b^{2}c^{2})^{k-1}R^{2}(X,Y), \;\; \textit{for}\;\; p = 2k
\end{array}
\right.
\end{equation*}
for any vector fields $X$ and $Y$  on $M$, where  $k\geq 2$ and $b^{2}=|X|^{2}|Y|^{2}-g(X,Y)^{2}$.
\end{lemma}

\begin{theorem}\label{th_6}
Let $(M^{2m}, \varphi, g)$ be a para-K\"{a}hler-Norden manifold of constant curvature $c\neq0$, $T^{\varphi}_{1}M$ its $\varphi$-unit tangent bundle equipped with the $\varphi$-Sasaki metric,  $\Gamma$ be a non-vertical geodesic on $T^{\varphi}_{1}M$ and $k_{1}, \ldots, k_{2m-1}$ the Frenet curvatures of the projected curve $\gamma=\pi\circ \Gamma$. If $k_{1}\neq0$ and $k_{2}\neq0$ then $k_{i}=0$ for $i\geq3$.
\end{theorem}

\begin{proof} 
From the Theorem \ref{th_3}, all Frenet curvatures  of the projected curve $\gamma=\pi\circ \Gamma$ are constants, and using \eqref{eq_13} we have 
\begin{eqnarray}\label{eq_16}
\gamma_{t}^{(4)}=-(1-\rho^{2})^{2}k_{1}(k_{1}^{2}+k_{2}^{2})\nu_{2}+(1-\rho^{2})^{2}k_{1}k_{2}k_{3}\nu_{4}.
\end{eqnarray}
On the other hand, from \eqref{eq_7}, Lemma \ref{lem_2} and \eqref{eq_12}, we have
\begin{equation}\label{eq_17}
\gamma_{t}^{(4)}=R^{3}(\xi_{t}^{\prime}, \varphi \xi)\gamma_{t}^{\prime}=-b^{2}c^{2}R(\xi_{t}^{\prime}, \varphi \xi)\gamma_{t}^{\prime}=-b^{2}c^{2}\gamma_{t}^{\prime\prime}=-b^{2}c^{2}(1-\rho^{2})k_{1}\nu_{2},
\end{equation}
If $k_{1}\neq0$ and $k_{2}\neq0$, we get
\begin{eqnarray}\label{eq_18}
(b^{2}c^{2}-(1-\rho^{2})(k_{1}^{2}+k_{2}^{2}))\nu_{2}+(1-\rho^{2})k_{2}k_{3}\nu_{4}=0.
\end{eqnarray}
Therefore, we have $k_{3} = 0$, and $b^{2}c^{2}=(1-\rho^{2})(k_{1}^{2}+k_{2}^{2})$.i.e. $b^{2}=$ const.\\
By continuing the process, we obtain $k_{i}=0$ for $i\geq3$.
\end{proof}
%--------------------------------------------------------------------------------

\section{$F$-geodesics on tangent bundle with the $\varphi$-Sasaki metric}
Let $(M^{m}, g)$ be an Riemannian manifold. A magnetic field is a closed $2$-form $\Omega$ on $(M^{m}, g)$ and the Lorentz force of a magnetic field $\Omega$ on $(M^{m},g)$ is a $(1, 1)$-tensor field $\Phi$ given by  
\begin{eqnarray}\label{eq_19}
g(\Phi(X),Y)=\Omega(X,Y),
\end{eqnarray} 
for any vector fields $X$ and $Y$ on $M$. The magnetic trajectories of $\Omega$ (or simply a magnetic curve) with strength $q\in \mathbb{R}$ are the curves $\gamma$ on $M$ that satisfy the Lorentz equation
\begin{eqnarray}\label{eq_20}
\gamma_{t}^{\prime\prime}=q\Phi \gamma_{t}^{\prime},
\end{eqnarray} 
where $\nabla$ is the Levi-Civita connection of $g$. The Lorentz equation generalizes the equation satisfied by the geodesics of $(M^{m}, g)$, namely $\gamma_{t}^{\prime\prime}= 0$.

Let $F$ be a $(1,1)$-tensor field on $(M^{m}, g)$. A curve $\gamma$ on $M$ is called $F$-planar if its speed remains, under parallel translation along the curve $\gamma$, in the distribution generated by the vector $\gamma_{t}^{\prime}$ and $F \gamma_{t}^{\prime}$ along $\gamma$. This is equivalent to the fact that the tangent vector $\gamma_{t}^{\prime}$ satisfies
\begin{eqnarray}\label{eq_21}
\gamma_{t}^{\prime\prime}=\varrho_{1}(t)\gamma_{t}^{\prime} +\varrho_{2}F \gamma_{t}^{\prime},
\end{eqnarray} 
where $\varrho_{1}$ and $\varrho_{2}$ are some functions of the parameter $t$, see \cite{M.S,H.M}. The $F$-planar curves generalize the magnetic curves and therefore, the geodesics. More precisely, when $F=\Phi$ is a Lorentz force on the manifold, $\varrho_{1}=0$ and $\varrho_{2}$ is a constant, we obtain the magnetic trajectories corresponding to $\Phi$ with strength $\varrho_{2}$. In the absence of $F$, one gets the geodesics, see \cite{Nis}.\\
We say that a curve $\gamma$ on $M$ is an $F$-geodesic if $\gamma$ satisfies:
\begin{eqnarray}\label{eq_22}
\gamma_{t}^{\prime\prime}=F \gamma_{t}^{\prime},
\end{eqnarray} 
One can see that an $F$-geodesic is an $F$-planar curve, but in general an $F$-planar curve is not always an $F$-geodesic, see \cite{B.D}. According to \eqref{eq_22}, the Lorentz equation \eqref{eq_20} expresses the relation satisfied by an $F$-geodesic on $M$, where  $F = q\Phi$,  see \cite{K.Y,Mik}.

\subsection{$F$-geodesics on tangent bundle with the $\varphi$-Sasaki metric}\,\\
Let $\widetilde{\nabla}$ be the Levi-Civita connection of $\varphi$-Sasaki metric on tangent bundle $TM$, given in the Theorem \ref{th_0}.	

\begin{theorem}\label{th_8}
Let $(M^{2m}, \varphi, g)$ be a para-K\"{a}hler-Norden manifold, $TM$ its tangent bundle equipped with the $\varphi$-Sasaki metric $g^{\varphi}$ and $F$ be a $(1,1)$-tensor field on $M$. A curve $\Gamma =(\gamma(t), \xi(t))$ on $TM$ is an ${}^{H}\!F$-planar with respect to $\widetilde{\nabla}$ if and only if the 
\begin{equation*}
\left\{\begin{array}{lll}
\gamma_{t}^{\prime\prime}=R(\xi_{t}^{\prime},\varphi\xi)\gamma_{t}^{\prime}+\varrho_{1}\gamma_{t}^{\prime}+\varrho_{2} F\gamma_{t}^{\prime}\\
\xi_{t}^{\prime\prime}=\varrho_{1}\xi_{t}^{\prime}+\varrho_{2} F\xi_{t}^{\prime}
\end{array}\right.
\end{equation*}
\end{theorem}

\begin{proof}
Using Theorem \ref{th_0} and \eqref{eq_1}, we find	
\begin{eqnarray}\label{eq_23}
\widetilde{\nabla}_{\Gamma_{t}^{\prime}}\Gamma_{t}^{\prime} &=& \widetilde{\nabla}_{\displaystyle({}^{H}\!\gamma_{t}^{\prime} + {}^{V}\!\xi_{t}^{\prime})}({}^{H}\!\gamma_{t}^{\prime} + {}^{V}\!\xi_{t}^{\prime}) \nonumber\\
&=&\widetilde{\nabla}_{\displaystyle{}^{H}\gamma_{t}^{\prime}}{}^{H}\gamma_{t}^{\prime} +\widetilde{\nabla}_{\displaystyle{}^{H}\!\gamma_{t}^{\prime}}{}^{V}\!\xi_{t}^{\prime}+\widetilde{\nabla}_{{}^{V}\!\xi_{t}^{\prime}}{}^{H}\!\gamma_{t}^{\prime}+\widetilde{\nabla}_{{}^{V}\!\xi_{t}^{\prime}}{}^{V}\!\xi_{t}^{\prime} \nonumber\\
&=&{}^{H}\!(\gamma_{t}^{\prime\prime}+R(\varphi \xi,\xi_{t}^{\prime})\gamma_{t}^{\prime})+{}^{V}\!\xi_{t}^{\prime\prime}
\end{eqnarray} 
On the other hand,
\begin{eqnarray}\label{eq_24}
\widetilde{\nabla}_{\Gamma_{t}^{\prime}}\Gamma_{t}^{\prime}&=&\varrho_{1}\Gamma_{t}^{\prime} +\varrho_{2}{}^{H}\!F \Gamma_{t}^{\prime}\nonumber\\
&=&\varrho_{1}({}^{H}\!\gamma_{t}^{\prime} + {}^{V}\!\xi_{t}^{\prime}) +\varrho_{2}{}^{H}\!F({}^{H}\!\gamma_{t}^{\prime} + {}^{V}\!\xi_{t}^{\prime})\nonumber\\
&=&{}^{V}\!\varrho_{1}{}^{H}\!\gamma_{t}^{\prime} +{}^{V}\!\varrho_{2}{}^{H}\!F {}^{H}\!\gamma_{t}^{\prime}+{}^{V}\!\varrho_{1}{}^{V}\!\xi_{t}^{\prime} +{}^{V}\!\varrho_{2}{}^{H}\!F {}^{V}\!\xi_{t}^{\prime}\nonumber\\
&=&{}^{H}\!(\varrho_{1}\gamma_{t}^{\prime}+\varrho_{2}F \gamma_{t}^{\prime})+{}^{V}\!(\varrho_{1}\xi_{t}^{\prime} +\varrho_{2}F \xi_{t}^{\prime}).
\end{eqnarray} 
From \eqref{eq_23} and \eqref{eq_24}, the result immediately follows.
\end{proof}

\begin{corollary}\label{co_5}
Let $(M^{2m}, \varphi, g)$ be a para-K\"{a}hler-Norden manifold and $TM$ its tangent bundle equipped with the $\varphi$-Sasaki metric $g^{\varphi}$. A curve $\Gamma =(\gamma(t), \xi(t))$ on $TM$ is an ${}^{H}\!\varphi$-planar with respect to $\widetilde{\nabla}$ if and only if the 
\begin{equation*}
\left\{\begin{array}{lll}
\gamma_{t}^{\prime\prime}=R(\xi_{t}^{\prime},\varphi\xi)\gamma_{t}^{\prime}+\varrho_{1}\gamma_{t}^{\prime}+\varrho_{2} \varphi\gamma_{t}^{\prime}\\
\xi_{t}^{\prime\prime}=\varrho_{1}\xi_{t}^{\prime}+\varrho_{2} \varphi\xi_{t}^{\prime}
\end{array}\right.
\end{equation*}
\end{corollary}

In the particular case when $\varrho_{1} = 0$ and $\varrho_{2} = 1$ in the Theorem \ref{th_8}, we obtain the following result.

\begin{theorem}\label{th_9}
Let $(M^{2m}, \varphi, g)$ be a para-K\"{a}hler-Norden manifold, $TM$ its tangent bundle equipped with the $\varphi$-Sasaki metric $g^{\varphi}$ and $F$ be a $(1,1)$-tensor field on $M$. A curve $\Gamma =(\gamma(t), \xi(t))$ on $TM$ is an ${}^{H}\!F$-geodesic with respect to $\widetilde{\nabla}$ if and only if the 
\begin{equation*}
\left\{\begin{array}{lll}
\gamma_{t}^{\prime\prime}=R(\xi_{t}^{\prime},\varphi\xi)\gamma_{t}^{\prime}+ F\gamma_{t}^{\prime}\\
\xi_{t}^{\prime\prime}=F\xi_{t}^{\prime}
\end{array}\right.
\end{equation*}
\end{theorem}

\begin{corollary}\label{co_6}
Let $(M^{2m}, \varphi, g)$ be a para-K\"{a}hler-Norden manifold and $TM$ its tangent bundle equipped with the $\varphi$-Sasaki metric $g^{\varphi}$. A curve $\Gamma =(\gamma(t), \xi(t))$ on $TM$ is an ${}^{H}\!\varphi$-geodesic with respect to $\widetilde{\nabla}$ if and only if the 
\begin{equation*}
\left\{\begin{array}{lll}
\gamma_{t}^{\prime\prime}=R(\xi_{t}^{\prime},\varphi\xi)\gamma_{t}^{\prime}+ \varphi\gamma_{t}^{\prime}\\
\xi_{t}^{\prime\prime}=\varphi\xi_{t}^{\prime}
\end{array}\right.
\end{equation*}
\end{corollary}

\begin{theorem}\label{th_10}
Let $(M^{2m}, \varphi, g)$ be a para-K\"{a}hler-Norden manifold, $TM$ its tangent bundle equipped with the $\varphi$-Sasaki metric $g^{\varphi}$ and $F$ be a $(1,1)$-tensor field on $M$. If $\Gamma =(\gamma(t), \xi(t))$ is a horizontal lift of a curve $\gamma$, then $\Gamma$ is an ${}^{H}\!F$-planar curve (resp., ${}^{H}\!F$-geodesic) if and only if $\gamma$ is an $F$-planar curve (resp., $F$-geodesic).
\end{theorem}

\begin{proof} Let $\gamma$ be an $F$-planar with respect to $\nabla$ on $M$, i.e. $\gamma$ satisfies	
\begin{eqnarray*}
\gamma_{t}^{\prime\prime}=\varrho_{1}\gamma_{t}^{\prime} +\varrho_{2}F \gamma_{t}^{\prime},
\end{eqnarray*} 
where $\varrho_{1}$ and $\varrho_{2}$ are some functions of the parameter $t$. Since $\Gamma =(\gamma(t), \xi(t))$ is a horizontal lift of a curve $\gamma$ then $\xi_{t}^{\prime}=0$
and  $\Gamma_{t}^{\prime}={}^{H}\!\gamma_{t}^{\prime}$.\\	
On the other hand,
\begin{eqnarray*}
\widetilde{\nabla}_{\Gamma_{t}^{\prime}}\Gamma_{t}^{\prime}&=&{}^{H}\!\gamma_{t}^{\prime\prime}\\
&=&{}^{H}\!(\varrho_{1}\gamma_{t}^{\prime}+\varrho_{2}F \gamma_{t}^{\prime})\\
&=&{}^{V}\!\varrho_{1}{}^{H}\!\gamma_{t}^{\prime} +{}^{V}\!\varrho_{2}{}^{H}\!F {}^{H}\!\gamma_{t}^{\prime}\\
&=&{}^{V}\!\varrho_{1}\Gamma_{t}^{\prime}+{}^{V}\!\varrho_{2}{}^{H}\!F \Gamma_{t}^{\prime}.
\end{eqnarray*} 
i.e. $\Gamma$ be an ${}^{H}\!F$-planar with respect to $\widetilde{\nabla}$. In the case of $\varrho_{1} = 0$ and $\varrho_{2} = 1$, we
get that $\Gamma$ is an ${}^{H}\!F$-geodesic if and only $\gamma$ is an $F$-geodesic.
\end{proof}

\begin{corollary}\label{co_7}
Let $(M^{2m}, \varphi, g)$ be a para-K\"{a}hler-Norden manifold, $TM$ its tangent bundle equipped with the $\varphi$-Sasaki metric $g^{\varphi}$. If $\Gamma =(\gamma(t), \xi(t))$ is a horizontal lift of a curve $\gamma$, then $\Gamma$ is an ${}^{H}\!\varphi$-planar curve (resp., ${}^{H}\!\varphi$-geodesic) if and only if $\gamma$ is an $\varphi$-planar curve (resp., $\varphi$-geodesic).
\end{corollary}

\begin{example}
Let $(\mathbb{R}^{2},\varphi,g)$ be a para-K\"{a}hler-Norden manifold such that
$$g= dx^{2}+dy^{2}, \quad \varphi=\left( \begin{array}{ccc}
1  & 0\\
0 &-1 
\end{array} \right).$$
Let $\Gamma =(\gamma(t), \xi(t))$ such that  $\gamma(t)=(x(t), y(t))$ and  $\xi(t) =(u(t),v(t))$\\
$1)$ \,$\Gamma$  is an ${}^{H}\!\varphi$-geodesic if and only if the 
\begin{eqnarray*}
\left\{\begin{array}{lll}
\gamma_{t}^{\prime\prime}= \varphi\gamma_{t}^{\prime}\\
\xi_{t}^{\prime\prime}=\varphi\xi_{t}^{\prime}
\end{array}\right.\Leftrightarrow
\left\{\begin{array}{lll}
x^{\prime\prime}= x^{\prime}\\
y^{\prime\prime}= -y^{\prime}\\
u^{\prime\prime}= u^{\prime}\\
v^{\prime\prime}= -v^{\prime}
\end{array}\right.\Leftrightarrow
\left\{\begin{array}{lll}
x(t)= k_{1}e^{t}+k_{2}\\
y(t)= k_{3}e^{-t}+k_{4}\\
u(t)= k_{5}e^{t}+k_{6}\\
v(t)= k_{7}e^{-t}+k_{8}
\end{array}\right.
\end{eqnarray*}
where $k_{i}$, $i=\overline{1,8}$  are real constants.\\
$2)$ \, $\Gamma$  is an ${}^{H}\!\varphi$-planar if and only if the 
\begin{eqnarray*}
\left\{\begin{array}{lll}
\gamma_{t}^{\prime\prime}=\varrho_{1}\gamma_{t}^{\prime}+ \varrho_{2}\varphi\gamma_{t}^{\prime}\\
\xi_{t}^{\prime\prime}=\varrho_{1}\xi_{t}^{\prime}+ \varrho_{2}\varphi\xi_{t}^{\prime}
\end{array}\right.&\Leftrightarrow&
\left\{\begin{array}{lll}
x^{\prime\prime}= (\varrho_{1}+ \varrho_{2})x^{\prime}\\
y^{\prime\prime}= (\varrho_{1}- \varrho_{2})y^{\prime}\\
u^{\prime\prime}= (\varrho_{1}+ \varrho_{2})u^{\prime}\\
v^{\prime\prime}= (\varrho_{1}- \varrho_{2})v^{\prime}
\end{array}\right.\\
&\Leftrightarrow&\left\{\begin{array}{lll}
x(t)= \varepsilon_{1}\int e^{\int (\varrho_{1}+\varrho_{2})dt}dt\\
y(t)=\varepsilon_{2}\int e^{\int (\varrho_{1}-\varrho_{2})dt}dt\\
u(t)= \varepsilon_{3}\int e^{\int (\varrho_{1}+\varrho_{2})dt}dt\\
v(t)=\varepsilon_{4}\int e^{\int (\varrho_{1}-\varrho_{2})dt}dt
\end{array}\right.
\end{eqnarray*}
where $\varepsilon_{1}=\varepsilon_{2}=\varepsilon_{3}=\varepsilon_{4}=\pm1$.\\
For example: $\varrho_{1}(t)= \dfrac{1}{t+1}$ and $\varrho_{2}(t)= \dfrac{1}{t-1}$, we find
$$\left\{\begin{array}{lll}
x(t)= a_{1}t^{3}-3a_{1}t+a_{2}\\
y(t)=a_{3}\ln (t+1)^{2}+a_{3}t+a_{4}\\
u(t)= b_{1}t^{3}-3b_{1}t+b_{2}\\
v(t)=b_{3}\ln (t+1)^{2}+b_{3}t+b_{4}
\end{array}\right.$$
where $a_{i}$, $b_{i}$, $i=\overline{1,4}$  are real constants, then \\
$\Gamma=(a_{1}t^{3}-3a_{1}t+a_{2}, a_{3}\ln (t+1)^{2}+a_{3}t+a_{4}, b_{1}t^{3}-3b_{1}t+b_{2}, b_{3}\ln (t+1)^{2}+b_{3}t+b_{4})$  is an ${}^{H}\!\varphi$-planar on $T\mathbb{R}^{2}$.
\end{example}

\begin{example}
Let $(\mathbb{R}^{2},\varphi,g)$ be a para-K\"{a}hler-Norden manifold such that
$$g= x^{2}dx^{2}+y^{2}dy^{2}, \quad \varphi=\left( \begin{array}{ccc}
0  & \dfrac{y}{x}\\
\dfrac{x}{y} &0 
\end{array} \right)\quad and \quad F=\left( \begin{array}{ccc}
a & 0\\
0 & b 
\end{array} \right),\; a,b\in \mathbb{R}.$$
The non-null Christoffel symbols of the Riemannian connection are:
$$\Gamma_{11}^{1} =\dfrac{1}{x}\;,\;\; \Gamma_{22}^{2} = \dfrac{1}{y}.$$
Let $\Gamma =(\gamma(t), \xi(t))$ be a horizontal lift of a curve $\gamma$, such that  $\gamma(t)=(x(t), y(t))$ and  $\xi(t) =(u(t),v(t))$ then $\xi_{t}^{\prime}=0$, from \eqref{eq_II} we have,
$$\dfrac{d\xi^{h}}{dt}+\sum_{i,j=1}^{2}\frac{d\gamma^{j}}{dt}\xi^{i}\Gamma_{ij}^{h}=0 \Leftrightarrow \left\{\begin{array}{lll}
u^{\prime} + \dfrac{x^{\prime}}{x}u=0\\
v^{\prime} + \dfrac{y^{\prime}}{y}v=0
\end{array} \right.\Leftrightarrow \left\{\begin{array}{lll}
u(t)=\dfrac{k_{1}}{x(t)}\\
v(t)=\dfrac{k_{2}}{y(t)}
\end{array} \right.$$
where $k_{1}, k_{2}$ are real constants.\\
$\gamma$ is an $F$-geodesic if and only if $\gamma_{t}^{\prime\prime}=F \gamma_{t}^{\prime}$,  from \eqref{eq_I} we have 
$$\left\{\begin{array}{lll}
x^{\prime\prime} + \dfrac{(x^{\prime})^{2}}{x}=ax^{\prime}\\
y^{\prime\prime} + \dfrac{(y^{\prime})^{2}}{y}=by^{\prime}
\end{array} \right.\Leftrightarrow \left\{\begin{array}{lll}
x(t)=\varepsilon_{1}\sqrt{c_{1}e^{at}+c_{2}}\\
y(t)=\varepsilon_{2}\sqrt{c_{3}e^{bt}+c_{4}}
\end{array} \right.$$
where $c_{1}, c_{2},  c_{3}, c_{4}$ are real constants and $\varepsilon_{1}=\varepsilon_{2}=\pm1$.\\
The horizontal lift $\Gamma =(\varepsilon_{1}\sqrt{c_{1}e^{at}+c_{2}}, \varepsilon_{2}\sqrt{c_{3}e^{bt}+c_{4}}, \dfrac{\varepsilon_{1}k_{1}}{\sqrt{c_{1}e^{at}+c_{2}}}, \dfrac{\varepsilon_{2}k_{2}}{\sqrt{c_{3}e^{bt}+c_{4}}})$ is an  ${}^{H}\!F$-geodesic on $T\mathbb{R}^{2}$.\\
$\gamma$ is an $F$-planar if and only if $\gamma_{t}^{\prime\prime}=\varrho_{1}\gamma_{t}^{\prime} +\varrho_{2}F \gamma_{t}^{\prime}$, where $\varrho_{1}$ and $\varrho_{2}$ are some functions of the parameter $t$, from \eqref{eq_I} we have 
$$\left\{\begin{array}{lll}
x^{\prime\prime} + \dfrac{(x^{\prime})^{2}}{x}=ax^{\prime}\\	y^{\prime\prime} + \dfrac{(y^{\prime})^{2}}{x}=by^{\prime}
\end{array} \right.\Leftrightarrow \left\{\begin{array}{lll}
x(t)=\varepsilon_{1}\sqrt{c_{1}e^{(\varrho_{1}+a\varrho_{2}t)}+c_{2}}\\
y(t)=\varepsilon_{1}\sqrt{c_{3}e^{(\varrho_{2}+b\varrho_{2}t)}+c_{4}}
\end{array} \right.$$
where $c_{1}, c_{2},  c_{3}, c_{4}$ are real constants and $\varepsilon_{1}=\varepsilon_{2}=\pm1$.\\
The horizontal lift 
$$\Gamma =(\varepsilon_{1}\sqrt{c_{1}e^{(\varrho_{1}+a\varrho_{2}t)}+c_{2}}, \varepsilon_{2}\sqrt{c_{3}e^{(\varrho_{1}+b\varrho_{2}t)}+c_{4}}, \dfrac{\varepsilon_{1}k_{1}}{\sqrt{c_{1}e^{(\varrho_{1}+a\varrho_{2}t)}+c_{2}}}, \dfrac{\varepsilon_{2}k_{2}}{\sqrt{c_{3}e^{(\varrho_{1}+b\varrho_{2}t)}+c_{4}}})$$ is an  ${}^{H}\!F$-planar on $T\mathbb{R}^{2}$.
\end{example}

\subsection{$F$-geodesics on $\varphi$-unit tangent bundle with the $\varphi$-Sasaki metric}\,\\
Let $\widehat{\nabla}$ be the Levi-Civita connection of $\varphi$-Sasaki metric on $\varphi$-unit tangent bundle $T^{\varphi}_{1}M$, given in the Theorem \ref{th_1}.
\begin{theorem}\label{th_11}
Let $(M^{2m}, \varphi, g)$ be a para-K\"{a}hler-Norden manifold and $T^{\varphi}_{1}M$ its $\varphi$-unit tangent bundle equipped with the $\varphi$-Sasaki metric and $F$ be a $(1,1)$-tensor field on $M$. A curve $\Gamma =(\gamma(t), \xi(t))$ on $T^{\varphi}_{1}M$ is an ${}^{H}\!F$-planar with respect to $\widehat{\nabla}$ if and only if the 
\begin{equation*}
\left\{\begin{array}{lll}
\gamma_{t}^{\prime\prime}=R(\xi_{t}^{\prime},\varphi\xi)\gamma_{t}^{\prime}+\varrho_{1}\gamma_{t}^{\prime}+\varrho_{2} F\gamma_{t}^{\prime}\\
\xi_{t}^{\prime\prime}=\varrho_{1}\xi_{t}^{\prime}+\varrho_{2} F\xi_{t}^{\prime}-\rho^{2}\xi
\end{array}\right.
\end{equation*}
where $\rho = const$ and $0\leq \rho \leq1$.
\end{theorem}

\begin{proof}
Using Theorem \ref{th_1} and \eqref{eq_2}, we find	
\begin{eqnarray}\label{eq_25}
\widehat{\nabla}_{\Gamma_{t}^{\prime}}\Gamma_{t}^{\prime} &=& \widehat{\nabla}_{\displaystyle({}^{H}\!\gamma_{t}^{\prime} + {}^{V}\!\xi_{t}^{\prime})}({}^{H}\!\gamma_{t}^{\prime} + {}^{T}\!\xi_{t}^{\prime}) \nonumber\\
&=&\widehat{\nabla}_{\displaystyle{}^{H}\gamma_{t}^{\prime}}{}^{H}\gamma_{t}^{\prime} +\widehat{\nabla}_{\displaystyle{}^{H}\!\gamma_{t}^{\prime}}{}^{T}\!\xi_{t}^{\prime}+\widehat{\nabla}_{{}^{T}\!\xi_{t}^{\prime}}{}^{H}\!\gamma_{t}^{\prime}+\widehat{\nabla}_{{}^{T}\!\xi_{t}^{\prime}}{}^{T}\!\xi_{t}^{\prime} \nonumber\\
&=&{}^{H}\!(\gamma_{t}^{\prime\prime}+R(\varphi \xi,\xi_{t}^{\prime})\gamma_{t}^{\prime})+{}^{T}\!\xi_{t}^{\prime\prime} \nonumber\\
&=&{}^{H}\!(\gamma_{t}^{\prime\prime}+R(\varphi \xi,\xi_{t}^{\prime})\gamma_{t}^{\prime})+{}^{V}\!(\xi_{t}^{\prime\prime}-g(\xi_{t}^{\prime\prime}, \varphi \xi)\xi).
\end{eqnarray} 
On the other hand,
\begin{eqnarray*}
\widehat{\nabla}_{\Gamma_{t}^{\prime}}\Gamma_{t}^{\prime}&=&\varrho_{1}\Gamma_{t}^{\prime} +\varrho_{2}{}^{H}\!F \Gamma_{t}^{\prime}\nonumber\\
&=&\varrho_{1}({}^{H}\!\gamma_{t}^{\prime} + {}^{T}\!\xi_{t}^{\prime}) +\varrho_{2}{}^{H}\!F({}^{H}\!\gamma_{t}^{\prime} + {}^{T}\!\xi_{t}^{\prime})\nonumber\\
&=&\varrho_{1}({}^{H}\!\gamma_{t}^{\prime} + {}^{V}\!\xi_{t}^{\prime}-g(\xi_{t}^{\prime}, \varphi \xi){}^{V}\!\xi) +\varrho_{2}{}^{H}\!F({}^{H}\!\gamma_{t}^{\prime} + {}^{V}\!\xi_{t}^{\prime}-g(\xi_{t}^{\prime}, \varphi \xi){}^{V}\!\xi).
\end{eqnarray*}
From \eqref{eq_3}, we have
\begin{eqnarray}\label{eq_26}
\widehat{\nabla}_{\Gamma_{t}^{\prime}}\Gamma_{t}^{\prime}&=&{}^{V}\!\varrho_{1}{}^{H}\!\gamma_{t}^{\prime} +{}^{V}\!\varrho_{2}{}^{H}\!F {}^{H}\!\gamma_{t}^{\prime}+{}^{V}\!\varrho_{1}{}^{V}\!\xi_{t}^{\prime} +{}^{V}\!\varrho_{2}{}^{H}\!F {}^{V}\!\xi_{t}^{\prime}\nonumber\\
&=&{}^{H}\!(\varrho_{1}\gamma_{t}^{\prime}+\varrho_{2}F \gamma_{t}^{\prime})+{}^{V}\!(\varrho_{1}\xi_{t}^{\prime} +\varrho_{2}F \xi_{t}^{\prime}).
\end{eqnarray} 
From \eqref{eq_25}, \eqref{eq_26} and \eqref{eq_6}, the result immediately follows.
\end{proof}

\begin{corollary}\label{co_8}
Let $(M^{2m}, \varphi, g)$ be a para-K\"{a}hler-Norden manifold and $T^{\varphi}_{1}M$ its $\varphi$-unit tangent bundle equipped with the $\varphi$-Sasaki metric. A curve $\Gamma =(\gamma(t), \xi(t))$ on $T^{\varphi}_{1}M$ is an ${}^{H}\!\varphi$-planar with respect to $\widehat{\nabla}$ if and only if the 
\begin{equation*}
\left\{\begin{array}{lll}
\gamma_{t}^{\prime\prime}=R(\xi_{t}^{\prime},\varphi\xi)\gamma_{t}^{\prime}+\varrho_{1}\gamma_{t}^{\prime}+\varrho_{2} \varphi\gamma_{t}^{\prime}\\
\xi_{t}^{\prime\prime}=\varrho_{1}\xi_{t}^{\prime}+\varrho_{2} \varphi\xi_{t}^{\prime}-\rho^{2}\xi
\end{array}\right.
\end{equation*}
\end{corollary}

In the particular case when $\varrho_{1} = 0$ and $\varrho_{2} = 1$ in the Theorem \ref{th_11}, we obtain the following result.

\begin{theorem}\label{th_12}
Let $(M^{2m}, \varphi, g)$ be a para-K\"{a}hler-Norden manifold and $T^{\varphi}_{1}M$ its $\varphi$-unit tangent bundle equipped with the $\varphi$-Sasaki metric and $F$ be a $(1,1)$-tensor field on $M$. A curve $\Gamma =(\gamma(t), \xi(t))$ on $TM$ is an ${}^{H}\!F$-geodesic with respect to $\widehat{\nabla}$ if and only if the 
\begin{equation*}
\left\{\begin{array}{lll}
\gamma_{t}^{\prime\prime}=R(\xi_{t}^{\prime},\varphi\xi)\gamma_{t}^{\prime}+ F\gamma_{t}^{\prime}\\
\xi_{t}^{\prime\prime}=F\xi_{t}^{\prime}-\rho^{2}\xi
\end{array}\right.
\end{equation*}
\end{theorem}

\begin{corollary}\label{co_9}
Let $(M^{2m}, \varphi, g)$ be a para-K\"{a}hler-Norden manifold and $T^{\varphi}_{1}M$ its $\varphi$-unit tangent bundle equipped with the $\varphi$-Sasaki metric. A curve $\Gamma =(\gamma(t), \xi(t))$ on $T^{\varphi}_{1}M$ is an ${}^{H}\!\varphi$-geodesic with respect to $\widehat{\nabla}$ if and only if the 
\begin{equation*}
\left\{\begin{array}{lll}
\gamma_{t}^{\prime\prime}=R(\xi_{t}^{\prime},\varphi\xi)\gamma_{t}^{\prime}+ \varphi\gamma_{t}^{\prime}\\
\xi_{t}^{\prime\prime}=\varphi\xi_{t}^{\prime}-\rho^{2}\xi
\end{array}\right.
\end{equation*}
\end{corollary}

\begin{theorem}\label{th_13}
Let $(M^{2m}, \varphi, g)$ be a para-K\"{a}hler-Norden manifold and $T^{\varphi}_{1}M$ its $\varphi$-unit tangent bundle equipped with the $\varphi$-Sasaki metric and $F$ be a $(1,1)$-tensor field on $M$. If $\Gamma =(\gamma(t), \xi(t))$ is a horizontal lift of $\gamma$ and $\Gamma\in T^{\varphi}_{1}M$, then $\Gamma$ is an ${}^{H}\!F$-planar curve (resp., ${}^{H}\!F$-geodesic) if and only if $\gamma$ is an $F$-planar curve (resp., $F$-geodesic).
\end{theorem}

\begin{proof} Let $\gamma$ be an $F$-planar with respect to $\nabla$ on $M$, i.e. $\gamma$ satisfies	
\begin{eqnarray*}
\gamma_{t}^{\prime\prime}=\varrho_{1}\gamma_{t}^{\prime} +\varrho_{2}F \gamma_{t}^{\prime},
\end{eqnarray*} 
where $\varrho_{1}$ and $\varrho_{2}$ are some functions of the parameter $t$. Since $\Gamma =(\gamma(t), \xi(t))$ is a horizontal lift of a curve $\gamma$ then $\xi_{t}^{\prime}=0$
and  $\Gamma_{t}^{\prime}={}^{H}\!\gamma_{t}^{\prime}$.\\	
On the other hand,
\begin{eqnarray*}
\widehat{\nabla}_{\Gamma_{t}^{\prime}}\Gamma_{t}^{\prime}&=&{}^{H}\!\gamma_{t}^{\prime\prime}\\
&=&{}^{H}\!(\varrho_{1}\gamma_{t}^{\prime}+\varrho_{2}F \gamma_{t}^{\prime})\\
&=&{}^{V}\!\varrho_{1}{}^{H}\!\gamma_{t}^{\prime} +{}^{V}\!\varrho_{2}{}^{H}\!F {}^{H}\!\gamma_{t}^{\prime}\\
&=&{}^{V}\!\varrho_{1}\Gamma_{t}^{\prime}+{}^{V}\!\varrho_{2}{}^{H}\!F \Gamma_{t}^{\prime}.
\end{eqnarray*} 
i.e. $\Gamma$ be an ${}^{H}\!F$-planar with respect to $\widehat{\nabla}$. In the case of $\varrho_{1} = 0$ and $\varrho_{2} = 1$, we
get that $\Gamma$ is an ${}^{H}\!F$-geodesic if and only $\gamma$ is an $F$-geodesic.
\end{proof}

%------------------------------------------------------------------------------
%\section*{Acknowledgment.}
%We thank the anonymous reviewers for their insightful comments and suggestions that helped us improve the paper.
%------------------------------------------------------------------------------

\end{document}